\documentclass[12pt, reqno]{amsart}

\usepackage[margin=1.1in]{geometry}

\usepackage{amssymb,upref}
\usepackage{amsmath, amsthm}
\usepackage{enumerate, graphicx, float}
\usepackage{relsize, mathrsfs, upgreek, bbm, eufrak}

\usepackage{color}
\usepackage{esint}

\theoremstyle{plain}

\newtheorem{thm}{Theorem}[section]
\newtheorem{lem}[thm]{Lemma}
\newtheorem{coro}[thm]{Corollary}

\theoremstyle{remark}

\newtheorem{remark}[thm]{Remark}
\numberwithin{equation}{section}

\def\XXint#1#2#3{{\setbox0=\hbox{$#1{#2#3}{\int}$ }
\vcenter{\hbox{$#2#3$ }}\kern-.6\wd0}}

\newcommand{\bee}{\begin{equation}}
\newcommand{\eee}{\end{equation}}
\newcommand{\be}{\begin{equation*}}
\newcommand{\ee}{\end{equation*}}

\newcommand{\eps}{\varepsilon}

\newcommand{\na}{\mathbb{N}}

\newcommand{\rn}{\mathbb{R}^n}

\newcommand{\dist}[1]{{\rm dist}(#1)}

\newcommand{\dx}{\, dx}

\DeclareMathOperator*{\esssup}{ess\,sup\,}

\DeclareMathOperator*{\essinf}{ess\,inf\,}

\makeatletter
\@namedef{subjclassname@2020}{\textup{2020} Mathematics Subject Classification}
\makeatother

\begin{document}



\address{Diego Maldonado, Kansas State University, Department of Mathematics. 138 Cardwell Hall, Manhattan, KS-66506, USA.} \email{dmaldona@ksu.edu}

\title[Weak porosity in spaces of homogeneous type]{Weak porosity in spaces of homogeneous type}
\author[Diego Maldonado]{Diego Maldonado}

\thanks{Partially supported by Simons Foundation grant MPS-TSM-00007229.}

\date{\today}

\begin{abstract} Based on the theory of Muckenhoupt weights, short and conceptually simpler proofs are provided for the following two implications: (1) if $(X, d, \mu)$ is a space of homogeneous type where  $d$-balls are open sets and the Lebesgue differentiation theorem holds true and if $E \subset X$ is a weakly porous set whose maximal $E$-free hole function $\rho_{d, E}$ is doubling, then $\dist{\cdot, E}^{-\alpha} \in A_1(X, d, \mu)$ for some $\alpha > 0$; and (2) now without the assumption on the validity of Lebesgue's differentiation theorem, if $\dist{\cdot, E}^{-\alpha} \in A_1(X, d, \mu)$ for some $\alpha > 0$, then  $\rho_{d, E}$ is doubling. 
\end{abstract}

\maketitle

\section{Introduction and main results}

The property of weak porosity for subsets of $\rn$ was introduced and developed by T.~C.~Anderson, J. Lehrb\"ack, C. Mudarra, and A. V\"ah\"akangas in \cite{ALMV} as follows: a nonempty set $E \subset \rn$ is said to be \emph{weakly porous} if there are constants $\gamma, \sigma \in (0,1)$ such that for every cube $P \subset \rn$ there exists $N \in \na$ and pairwise disjoint subcubes $\{Q_j\}_{j=1}^N \in \mathcal{D}(P)$ (where $\mathcal{D}(P)$ stands for the family of dyadic subcubes of $P$) such that:
\begin{enumerate}[\bf{wp}(i)]
\item the cubes $Q_j$ are $E$-free, meaning $Q_j \cap E = \emptyset$ for $j=1, \ldots, N$, 
\item\label{cond:MP} $|Q_j| \geq \gamma |\mathcal{M}(P)|$ for $j=1, \ldots, N$, where $\mathcal{M}(P) \in \mathcal{D}(P)$ denotes a largest $E$-free dyadic subcube of $P$, that is, $\mathcal{M}(P) \cap E = \emptyset$ and if $R \in \mathcal{D}(P)$ satisfies $R \cap E = \emptyset$, then $|\mathcal{M}(P)| \geq |R|$ (such cube $\mathcal{M}(P)$ may not be unique, but one is fixed), and 
\item $\sum_{j=1}^N |Q_j| \geq \sigma |P|$.
\end{enumerate}
The class of weakly porous sets in $\rn$ strictly includes that of porous sets, which in turn includes the class of $\lambda$-Ahlfors regular sets with $0 < \lambda < n$, see \cite[Section 3]{ALMV}. 

As the main result in \cite{ALMV} stands a characterization of the weak porosity of $E \subset \rn$ in terms of its distance function  $\dist{x, E}:= \inf \{|x - y|: y \in E\}$ for $x \in \rn$; namely,  $E \subset \rn$ is weakly porous if and only if there exists $\alpha >0$ such that  $\dist{\cdot, E}^{-\alpha}$ belongs to the Muckenhoupt class $A_1(\rn)$, that is,
$$
\sup\limits_Q \left( \fint_Q \dist{x, E}^{-\alpha} \dx \right) \left(\essinf_Q  \dist{\cdot, E}^{-\alpha} \right)^{-1} < \infty,
$$
where the supremum is taken over all the cubes $Q \subset \rn$ (always with sides parallel to the coordinate axes) and $\fint_Q u:= \frac{1}{|Q|}\int_Q u$ denotes the average of a function $u$ over $Q$. 

In the larger contexts of complete metric spaces with a doubling measure and that of spaces of homogeneous type, the notion of weak porosity has been defined and developed by C. Mudarra in \cite{Mud25} and by H. Aimar, I. G\'omez, and I. G\'omez-Vargas in \cite{AGG1, IGVtesis}, respectively. More notation is in order. Given a set $X$ and $K \geq 1$, a function $d : X \times X \to [0, \infty)$ is called a \emph{$K$-quasi-distance} on $X$ if it satisfies the following three conditions:
\begin{enumerate}[\bf{QD}(i)]
\item\label{item:dxx:0} $d(x,y) = 0$ if and only if $x=y$; 
\item\label{item:K:quasi:sym} $d$ is symmetric, that is, $d(y,x) = d(x,y)$, for every $x, y \in X$; and
\item\label{item:K:quasi:triangle} $d$ satisfies a \emph{$K$-quasi-triangle inequality}, meaning, 
\begin{equation}\label{def:K}
d(x,y) \leq K ( d(x,z) + d(z,y) ) \quad \forall x, y, z \in X.
\end{equation}
\end{enumerate}
If $d$ is a $K$-quasi-distance on $X$, for some $K \geq 1$, then $(X, d)$ is called a \emph{quasi-metric space} and the $d$-ball centered at $x \in X$ with radius $r >0$ is $B(x,r):=\{y \in X: d(x,y) < r\}$. If $K=1$, a $1$-quasi-distance on $X$ is  a \emph{distance} on $X$ and $(X, d)$ is a \emph{metric space}. Given a quasi-metric space  $(X, d)$, let $\mu$ be a nonnegative measure defined on a $\sigma$-algebra of $X$ which contains all the $d$-balls, then $\mu$ is a \emph{doubling measure} on $(X, d)$ if there exists $C_\mu \geq 1$ such that
\begin{equation}\label{def:Cmu}
0 < \mu(B(x,2r)) \leq C_\mu \, \mu(B(x,r) < \infty, \quad \forall x \in X, r >0. 
\end{equation}
If $\mu$ is a doubling measure on  a quasi-metric space  $(X, d)$, the triple $(X, d, \mu)$ is called a \emph{space of homogeneous type with constants} $(K_d, C_\mu)$, where $K_d \geq 1$ is the quasi-triangle constant for $d$ in \eqref{def:K}.

Following  C. Mudarra in \cite[Definition 3.1]{Mud25} and H. Aimar, I. G\'omez, and I. G\'omez-Vargas in \cite[Definition 3.7]{AGG1}, a nonempty set $E \subset X$ is \emph{weakly porous} in a space of homogeneous type $(X, d, \mu)$ if there exist $\gamma, \sigma \in (0,1)$ such that for every $d$-ball $B=B(x,r)$ there is a collection of pairwise disjoint $d$-balls $\{B(x_j, r_j)\}_{j=1}^N$ satisfying
\begin{enumerate}[\bf{WP}(i)]
\item\label{cond:1} $B(x_j, r_j) \subset B \setminus E$ for every $j=1, \ldots, N$,
\item\label{cond:2} $r_j \geq \gamma \rho_{d, E}(B)$ for every $j=1, \ldots, N$; here $\rho_{d, E}(B)$, known as the \emph{maximal $E$-free hole function}, plays the role of $\mathcal{M}(P)$ from {\bf{wp}\eqref{cond:MP}}; more precisely, 
$$
\rho_{d, E}(B) := \sup \{0 < s < 2K_d r: \text{such that there exists } y \in X \text{ with } B(y,s) \subset B\setminus E\},
$$
(notice that {\bf{WP}{\bf{\eqref{cond:1}}}} implies $\rho_{d, E}(B) > 0$), 
\item\label{cond:3} $\sum_{j=1}^N \mu(B(x_j, r_j)) \geq \sigma \mu(B)$, and
\item\label{cond:4} $r_j \leq 2K_d r$ for every $j=1, \ldots, N$.
\end{enumerate}
The following doubling property for the maximal $E$-free hole function plays a key role in the characterization of weakly porous sets in a space of homogeneous type $(X, d, \mu)$: given $E \subset X$, the function $\rho_{d, E}$ is called \emph{doubling} if there exists $C_{\rho, E} \geq 1$ such that
\begin{equation}\label{rho:doubling}
\rho_{d, E}(B(x,2r)) \leq C_{\rho, E} \, \rho_{d, E}(B(x,r)), \quad \forall x \in X, r >0. 
\end{equation}
In the Euclidean case, the doubling property of $\mathcal{M}$ as in  {\bf{wp}\eqref{cond:MP}} follows from the weak porosity of $E \subset \rn$ as in \cite[Lemma 3.2]{ALMV}; however, as proved in \cite[Section 8]{Mud25}, that implication fails in general even for connected metric spaces with doubling Borel measures. 

Let $\dist{x, E}:=\inf\{d(x,y): y \in E\}$ for $x \in X$ denote the distance function of $E \subset X$. The equivalence between the statements
\begin{enumerate}[(A)]
\item $E \subset X$ is weakly porous and $\rho_{d, E}$ is doubling, and
\item there exists $\alpha >0$ such that $\dist{\cdot, E}^{-\alpha} \in A_1(X, d, \mu)$, that is,
\begin{equation*}
\sup\limits_{B \, d\text{-ball}} \left( \fint_B \dist{\cdot, E}^{-\alpha} \, d\mu \right) \left(\essinf_B  \dist{\cdot, E}^{-\alpha} \right)^{-1} < \infty,
\end{equation*}
\end{enumerate}
has been proved in \cite[Theorem 1.1]{Mud25} under the hypotheses that $(X, d)$ is a complete metric space with $\mu$ a doubling Borel measure on $X$; and in \cite[Theorem 1.1]{AGG1} under the hypothesis that the $d$-balls in $(X, d)$ are open sets (with respect to the topology they generate) also with $\mu$ a doubling Borel measure on $X$. 

Under the hypotheses that $d$-balls are open sets  plus the additional assumption that the Lebesgue differentiation theorem holds in $(X, d, \mu)$, in Section \ref{sec:thm:A:B} we provide a short, conceptually simple proof of the following theorem that yields the implication (A) $\Rightarrow$ (B).  
 
\begin{thm}\label{thm:main} Let $(X, d, \mu)$ be a space of homogenous type with constants $(K_d, C_\mu)$ where $d$-balls are open sets and the Lebesgue differentiation theorem holds true. If $E \subset X$ is a weakly porous set with some constants $\gamma, \sigma \in (0,1)$ with $\rho_{d, E}$ doubling with constant $C_{\rho, E}$, then there exists $\alpha > 0$ such that $\dist{\cdot, E}^{-\alpha} \in A_1(X, d, \mu)$, that is, 
\begin{equation}\label{dist:E:A1}
[\dist{\cdot, E}^{-\alpha}]_{A_1}:=\sup\limits_{B \, d\text{-ball}} \left( \fint_B \dist{\cdot, E}^{-\alpha} \, d\mu \right) \left(\essinf_B  \dist{\cdot, E}^{-\alpha} \right)^{-1} < \infty,
\end{equation} 
where $\alpha$ and $[\dist{\cdot, E}^{-\alpha}]_{A_1}$ depend only in $K_d$, $C_\mu$, $C_{\rho, E}$, $\gamma$, and $\sigma$. 
\end{thm}

\begin{remark} In the aforementioned contexts, when proving the equivalence (A) $\Leftrightarrow$ (B) it is the implication (A) $\Rightarrow$ (B) that appears to be the more technically involved, as can be appreciated in  \cite[Section 5]{ALMV}, \cite[Sections 5 and 6]{Mud25}, and \cite[Section 5]{AGG1}.
\end{remark}

Regarding the implication (B) $\Rightarrow$ (A), in Section \ref{sec:A1:rho:doub} we provide a short proof for the following theorem

\begin{thm}\label{thm:A1:rho:doub} Let $(X, d, \mu)$ be a space of homogeneous type where the $d$-balls are open sets. Let $E \subset X$ be a nonempty set such that $\dist{\cdot, E}^{-\alpha} \in A_1(X, d, \mu)$ for some $\alpha >0$. Then, there exists a constant $C \geq 1$, depending only on $C_\mu$, $K_d$, $\alpha$, and $[\dist{\cdot, E}^{-\alpha}]_{A_1}$, such that
\begin{equation}\label{doub:rho}
\rho_{d, E}(B) \leq C \rho_{d, E}(\tfrac{1}{2}B), \quad \forall d\text{-ball } B. 
\end{equation}
\end{thm}

\section{Preliminaries and remarks}

Our proof of Theorem \ref{thm:main} hinges upon the following result that extends the well-known fact that weights in $\rn$ satisfying a reverse-H\"older inequality are $A_\infty(\rn)$ Muckenhoupt weights; namely:

\begin{lem}\label{lemma:RH>Ap} Let $(X, d, \mu)$ be a space of homogenous type with constants $(K_d, C_\mu)$ where the $d$-balls are open sets and the Lebesgue differentiation theorem holds true. Given a weight $w$ in $(X, d, \mu)$ (i.e., $w \in L^1_{\rm{loc}}(X, d, \mu)$ with $w \geq 0$ $\mu$-a.e. in $X$) such that:
\begin{enumerate}[(i)]
\item $w$ is doubling, that is, there exists $C_w \geq 1$ such that $w(B(x, 2r)) \leq C_w w(B(x,r))$ for every $x \in X$ and $r > 0$, and
\item $w$ satisfies a reverse-H\"older inequality for some $q > 1$, that is, 
\begin{equation}\label{w:RHq}
[w]_{RH_q}:= \sup\limits_{B \, d\text{-ball}} \left(\fint_B w^{q} \, d\mu \right)^{1/q}  \left(\fint_B w \, d\mu \right)^{-1} < \infty,
\end{equation}
\end{enumerate}
then there exists $p > 1$ (depending only on $K_d$, $C_\mu$, $C_w$, $q$ and $[w]_{RH_q}$) such that $w \in A_p(X, d, \mu)$, meaning
$$
[w]_{A_p}:= \sup\limits_{B \, d\text{-ball}} \left(\fint_B w  \, d\mu \right) \left(\fint_B w^{-1/(p-1)} \, d\mu \right)^{p-1} < \infty. 
$$
\end{lem}
\begin{remark}\label{rmk:idea:proof} The idea of the proof for (A) $\Rightarrow$ (B) is then rather transparent: we will show that (A) implies that $\dist{\cdot, E}$ is a doubling weight in $(X, d, \mu)$ that satisfies the reverse-H\"older inequality
\begin{equation}\label{def:RHinf}
[\dist{\cdot, E}]_{RH_\infty}:= \sup\limits_{B \, d\text{-ball}} \left(\esssup\limits_B  \dist{\cdot, E} \right)  \left(\fint_B \dist{\cdot, E} \, d\mu \right)^{-1} < \infty,
\end{equation}
thus \eqref{w:RHq} holds for any fixed $q > 1$. By  Lemma \ref{lemma:RH>Ap} there is $p> 1$ such that
$$
[\dist{\cdot, E}]_{A_p}:= \sup\limits_{B \, d\text{-ball}} \left(\fint_B \dist{\cdot, E} \, d\mu \right) \left(\fint_B \dist{\cdot, E}^{-1/(p-1)} \, d\mu \right)^{q-1} < \infty,
$$
and, by taking $\alpha:=1/(p-1) >0$, given a $d$-ball $B$ we have
\begin{align}\label{RH:Ap:A1}
\left(\fint_B \dist{\cdot, E}^{-\alpha} \, d\mu  \right)^{1/\alpha} & \leq \frac{[\dist{\cdot, E}]_{A_p}}{\left(\fint_B \dist{\cdot, E} \, d\mu \right)} \leq \frac{[\dist{\cdot, E}]_{A_p} [\dist{\cdot, E}]_{RH_\infty}}{\esssup\limits_B \dist{\cdot, E}}\\\nonumber
& = [\dist{\cdot, E}]_{A_p} [\dist{\cdot, E}]_{RH_\infty} \essinf\limits_B \dist{\cdot, E}^{-1}, 
\end{align}
which leads to \eqref{dist:E:A1} with $[\dist{\cdot, E}^{-\alpha}]_{A_1} \leq ([\dist{\cdot, E}]_{A_p} [\dist{\cdot, E}]_{RH_\infty})^\alpha$. 
\end{remark}

\begin{remark} A proof of Lemma \ref{lemma:RH>Ap} can be found, for instance, in \cite[p.3391]{IMS13} as well as in \cite[Chapter 1]{ST89}. Notice the misprint in the condition (4) on \cite[p.7]{ST89},  where the inequality ``$w_{S, \mu} \leq C \nu(S)/\mu(S)$'' should read ``$w_{S, \nu} \leq C \nu(S)/\mu(S)$.'' 
\end{remark}

\begin{remark}\label{rmk:LDT} The Lebesgue differentiation theorem is said to hold in a space homogeneous type $(X, d, \mu)$ if for every $f \in L^1_{\rm{loc}}(X, d, \mu)$ we have
$$
\lim\limits_{r \to 0^+} \fint_{B(x, r)} |f(y) - f(x)| \, d\mu(y) = 0, \quad \text{for } \mu\text{-a.e. } x \in X.
$$ 
By \cite[Theorem 3.14]{AlvaMitrea15}, the Lebesgue differentiation theorem holds in  $(X, d, \mu)$ if and only if there exists $p \in (0, \infty)$ such that the class of continuous functions with bounded support is dense in $L^{p}(X, d, \mu)$ if and only if $\mu$ is Borel-semiregular on $(X,d)$. If $Borel(X, d)$ denotes the smallest $\sigma$-algebra containing the open sets of $(X,d)$, then a measure $\mu$ defined on a $\sigma$-algebra $\Sigma$ of $X$ is called \emph{Borel-semiregular on} $(X,d)$ if $Borel(X, d) \subset \Sigma$ and for every $E \in \Sigma$ with $\mu(E) < \infty$ there exists $F \in Borel(X, d)$ such that $\mu((E \setminus F) \cup (F \setminus E)) =0$, see \cite[Section 3.3]{AlvaMitrea15}.
\end{remark}

\begin{remark}\label{rmk:mu:Ebar:0} Given a set $E \subset X$, the assumption on the openness of the $d$-balls in $(X, d, \mu)$ makes the function $x \mapsto \dist{x, E}$ $\mu$-measurable as well as $\overline{E} = \{x \in X: \dist{x,E}=0\}$, see \cite[Proposition 2.2]{AGG1}. In particular, $E$ is weakly porous if and only if $\overline{E}$ is weakly porous, both with the same constants $\gamma, \sigma \in (0,1)$. Also, the Lebesgue differentiation theorem implies that $\mu(\overline{E}) = 0$ if $\overline{E}$ is weakly porous. Indeed, if for some $x \in \overline{E}$ we had
\begin{equation}\label{density:E:x}
\lim\limits_{r \to 0^+} \frac{\mu(\overline{E} \cap B(x,r))}{\mu(B(x,r))} = 1,
\end{equation}
then by fixing $r > 0$ and letting $\{B(x_j, r_j)\}_{j=1}^N$ be the $d$-balls from the definition of weak porosity of $\overline{E}$ applied to $B(x,r)$ we would get
$$
\frac{\mu(\overline{E} \cap B(x,r))}{\mu(B(x,r))} = 1 - \frac{\mu(B(x,r) \setminus \overline{E})}{\mu(B(x,r))} \leq 1 - \frac{\mu(\cup_{j=1}^{N}B(x_j, r_j))}{\mu(B(x,r))} \leq 1 - \sigma,
$$
contradicting \eqref{density:E:x}. 
\end{remark}

\begin{remark} As mentioned, in the work \cite{AGG1} the only hypothesis assumed on the space of homogeneous type $(X, d, \mu)$ for the equivalence  (A) $\Leftrightarrow$ (B) is the openness of the $d$-balls. Thus, our present short proof of the implication  (A) $\Rightarrow$ (B) comes at the expense of the additional hypothesis on the validity of Lebesgue's differentiation theorem. On the other hand, these two hypotheses are satisfied in the contexts of \cite{ALMV} (Euclidean space with the Lebegue measure) and \cite{Mud25} (complete metric spaces with a doubling Borel regular measure). 
\end{remark}

Our proof of Theorem \ref{thm:A1:rho:doub} will be based on the following basic facts on Muckenhoupt weights. 

\begin{lem}\label{lem:basic:Ainf} Let $(X, d, \mu)$ be a space of homogeneous type. Then,
\begin{enumerate}[(i)]
\item\label{lem:basic:Ainf:1} if $w \in A_1(X, d, \mu)$, that is,
\begin{equation}\label{w:A1}
[w]_{A_1}:=\sup\limits_{B \, d\text{-ball}} \left( \fint_B w \, d\mu \right) \left(\essinf_B  w \right)^{-1} < \infty,
\end{equation}
and if $w^{-1} \in L^1_{\rm{loc}}(X, d, \mu)$, then $w^{-1}$ is a doubling weight with constant  $C_\mu^2 [w]_{A_1}$ that satisfies the $RH_\infty$ reverse-H\"older inequality
\begin{equation}\label{w:inv:RH:inft}
\esssup_B w^{-1} \leq [w]_{A_1} \left( \fint_B w^{-1} \, d\mu \right), \quad  \forall d\text{-ball } B;
\end{equation}
\item if $u \in RH_\infty(X, d, \mu)$, that is, 
$$
[u]_{RH_\infty}:= \sup\limits_{B \, d\text{-ball}} \left(\esssup\limits_B  u \right)  \left(\fint_B u \, d\mu \right)^{-1} < \infty,
$$
then given any $d$-ball $B$ and any $\mu$-measurable set $F \subset B$ with $\mu(F) >0$, we have
\begin{equation}\label{fint:F:fint:B:u}
\fint_F u \, d\mu \leq [u]_{RH_\infty} \fint_B u \, d\mu. 
\end{equation}
\end{enumerate}

\end{lem}

\begin{proof}  Given a $d$-ball $B$, from \eqref{w:A1} we get
\begin{align*}
\esssup_B w^{-1} \leq [w]_{A_1} \left( \fint_B w \, d\mu \right)^{-1} \leq  [w]_{A_1} \left( \fint_B w^{-1} \, d\mu \right),
\end{align*}
where the last inequality is just the harmonic-mean arithmetic-mean inequality, and \eqref{w:inv:RH:inft} follows whenever $w^{-1} \in L^1_{\rm{loc}}(X, d, \mu)$. Now, given a $d$-ball $B$,
\begin{align*}
\fint_{2B} w^{-1} \, d\mu \leq \esssup_{2B} w^{-1} \leq [w]_{A_1}  \left( \fint_{2B} w \, d\mu \right)^{-1} & \leq C_\mu [w]_{A_1}\left( \fint_{B} w \, d\mu \right)^{-1} \\ 
& \leq C_\mu [w]_{A_1}  \left( \fint_B w^{-1} \, d\mu \right),
\end{align*}
which implies that $w^{-1}$ is doubling with constant  $C_\mu^2 [w]_{A_1}$. Next,  if $u \in RH_\infty(X, d, \mu)$, given a $d$-ball $B$ and a $\mu$-measurable set $F \subset B$ with $\mu(F) >0$, we write
\begin{align*}
u(F) = \int_F u \, d\mu \leq (\esssup_F u) \mu(F)  \leq (\esssup_B u) \mu(F) \leq  [u]_{RH_\infty} \frac{u(B)}{\mu(B)} \mu(F)
\end{align*}
and  \eqref{fint:F:fint:B:u} follows. \end{proof}

Given a nonempty set $E \subset X$ and a $d$-ball $B$, let us define
$$
\sigma_{d,E}(B):=\esssup\limits_B \dist{\cdot, E}
$$
and $\sigma_{d,E}$ is called \emph{doubling} if there is $C \geq 1$ with 
$$
\sigma_{d,E}(2B) \leq C \sigma_{d,E}(B), \quad \forall d\text{-ball } B.
$$
\begin{remark}\label{rmk:sigma:doubling} Notice that if $\mu(\overline{E}) =0$, then $B \setminus \overline{E}$ is a nonempty open set and then $\sigma_{d,E}(B)>0$ (as well as $\rho_{d,E}(B) >0$). Also, if $\dist{\cdot, E} \in RH_\infty(X, d, \mu)$ with $\dist{\cdot, E}$ doubling, then $\sigma_{d,E}$ is also doubling, since $\dist{\cdot, E}(B)/\mu(B) \simeq  \sigma_{d,E}(B)$ for every $B$ whenever $\dist{\cdot, E} \in RH_\infty(X, d, \mu)$. 

\end{remark}

The next lemma is inspired by \cite[Lemma 3.4]{AGG1} and relates $\sigma_{d,E}$ and $\rho_{d,E}$.

\begin{lem}\label{lemma:rho:sigma} Let $(X, d, \mu)$ be a space of homogeneous type with constants $(K_d, C_\mu)$ where the $d$-balls are open sets. Let $E \subset X$ be a nonempty set with $\mu(\overline{E}) =0$. Then, 
\begin{enumerate}[(i)]
\item\label{item:dzE>tau} given a $d$-ball $B:=B(x,r)$, $C > K_d$, and a $d$-ball $B(y, \tau) \subset B \setminus E$, we have
\begin{equation}\label{dzE:tau}
\dist{z, E} > \left(\frac{1}{K_d} - \frac{1}{C}\right) \tau, \quad \forall z \in B(y, \tau/C),
\end{equation}

\item\label{item:rho<sigma} $\rho_{d,E}(B) \leq 2K_d \, \sigma_{d, E}(B)$ for every $d$-ball $B=B(x,r)$, 

\item\label{item:sigma<r} $\sigma_{d,E}(B) < 2K_d \, r$ for every $d$-ball $B=B(x,r)$ with  $B \cap E \ne \emptyset$, 

\item\label{item:sigma<rho:sigma} $\sigma_{d,E}(B) \leq \rho_{d, E}(B(x, K_d(\sigma_{d,E}(B) + r)))$ for every $d$-ball $B=B(x,r)$ with  $B \cap E \ne \emptyset$, which together with \eqref{item:sigma<r} implies
\begin{equation}\label{dist:<:rho:Kd}
\sigma_{d,E}(B) \leq \rho_{d, E}(B(x, K_d(2K_d  + 1) r)).
\end{equation} 
\end{enumerate}

\end{lem}

\begin{proof} \eqref{item:dzE>tau}: since $\dist{y, E} \geq \tau$ (due to $B(y, \tau) \subset B \setminus E$), for $e \in E$ and $z \in B(y, \tau/C)$ we write
\begin{align*}
\tau \leq \dist{y, E} \leq d(y,e) \leq K_d d(y, z) + K_d d(z,e) < \frac{K_d}{C} \tau + K_d d(z,e)
\end{align*}
and \eqref{dzE:tau} follows. \eqref{item:rho<sigma}: given $B=B(x,r)$ and $B(y, \tau) \subset B \setminus E$, \eqref{dzE:tau} with $C:=2K_d$ gives
$$
\dist{z, E} \geq \frac{\tau}{2K_d}, \quad \forall z \in B(y, \tau/(2K_d)),
$$
so that $\tau \leq 2K_d \sigma_{d,E} (B(y, \tau/(2K_d))) \leq 2K_d \sigma_{d,E} (B)$ and, by taking supremum in $\tau$, $\rho_{d,E}(B) \leq 2K_d \, \sigma_{d, E}(B)$.  \eqref{item:sigma<r}: given $B=B(x,r)$ with $B \cap E \ne \emptyset$, fix $e \in B \cap E$ so that for every  $z \in B(x,r)$ we have $\dist{z, E} \leq d(z, e) \leq K_d (d(x,z) + d(x, e)) < 2K_d r$ and, by taking supremum in $z$,  $\sigma_{d,E}(B) \leq 2K_d \, r$. \eqref{item:sigma<rho:sigma}: given $B=B(x,r)$ with $B \cap E \ne \emptyset$ and $0 < \eps < \sigma_{d,E}(B)$ let $y \in B$ such that $\dist{y, E} \geq  \sigma_{d,E}(B) - \eps$; in particular, $B(y, \sigma_{d,E}(B) - \eps) \cap E = \emptyset$. Also, since $y \in B(x,r)$, 
$$
B(y, \sigma_{d,E}(B) - \eps) \subset B(x, K_d(\sigma_{d,E}(B) - \eps + r)),
$$
and by \eqref{item:sigma<r}, $\sigma_{d,E}(B) < 2K_d \, r$; so that $0<\sigma_{d,E}(B) - \eps < 2K_d r$. Then the definition of $\rho_{d,E}$ yields
$$
\sigma_{d,E}(B) - \eps \leq \rho_{d,E}(B(x, K_d(\sigma_{d,E}(B) - \eps + r))),
$$
and we obtain \eqref{item:sigma<rho:sigma} by letting $\eps \to 0$. 
\end{proof}

\section{A short proof of Theorem \ref{thm:main} (weak porosity implies $A_1$)}\label{sec:thm:A:B}

\begin{lem}\label{lem:Ewp:rho:aver:dE} Let $(X, d, \mu)$ be a space of homogeneous type with constants $(K_d, C_\mu)$ where the $d$-balls are open sets and let $E \subset X$ be a weakly porous set with constants $\gamma, \sigma \in (0,1)$. Then,  there exists $C_1 \geq 1$, depending only on $K_d$ and $C_\mu$, such that
\begin{equation}\label{rho:average:d}
\rho_{d, E}(B) \leq \frac{C_1}{\gamma \sigma} \fint_B \dist{\cdot, E} \, d\mu, \quad \forall d\text{-ball } B.
\end{equation}
\end{lem}

\begin{proof} Given the pairwise disjoint balls $\{B(x_j, r_j)\}_{j=1}^N$ associated to $B$ from the definition of weak porosity and by taking $C:=2K_d$ in \eqref{dzE:tau} the conditions {\bf{WP}{\bf{\eqref{cond:1}}}}  and {\bf{WP}{\bf{\eqref{cond:2}}}} imply
\begin{equation*}
\dist{z, E} \geq \frac{r_j}{2K_d} \geq \frac{\gamma}{2K_d} \rho_{d, E}(B), \quad \forall z \in B(x_j, r_j/(2K_d)), j=1, \ldots, N.
\end{equation*}
Hence, 
\begin{align*}
& \int_B \dist{\cdot, E} \, d\mu \geq \int_{\cup_{j=1}^N B(x_j, r_j)} \dist{\cdot, E} \, d\mu = \sum\limits_{j=1}^N \int_{B(x_j, r_j)} \dist{\cdot, E} \, d\mu\\
& \geq \sum\limits_{j=1}^N \int_{B(x_j, r_j/(2 K_d))} \dist{\cdot, E} \, d\mu  \geq  \frac{\gamma}{2K_d} \rho_{d, E}(B) \sum\limits_{j=1}^N \mu(B(x_j, r_j/(2 K_d)) )\\
& \geq  \frac{\gamma c_1}{2K_d} \rho_{d, E}(B) \sum\limits_{j=1}^N \mu(B(x_j, r_j) )  \geq \frac{\gamma \sigma c_1}{2K_d} \rho_{d, E}(B) \mu(B) =: \frac{\gamma \sigma}{C_1} \rho_{d, E}(B) \mu(B),
\end{align*}
where for the last inequality we used {\bf{WP}{\bf{\eqref{cond:3}}}} and here $c_1 \in (0,1)$, depending only on $K_d$ and $C_\mu$, satisfies $\mu(B(x, R)) \geq c_1 \mu(B(x, 2K_d R)) $ for every $x \in X$ and  $R >0$, and \eqref{rho:average:d} follows. \end{proof}

Theorem \ref{thm:main} will follow as in Remark \ref{rmk:idea:proof} from Lemma \ref{lemma:RH>Ap} together with the next result.

\begin{thm}\label{thm:RH:doub} Let $(X, d, \mu)$ be a space of homogenous type with constants $(K_d, C_\mu)$ where the $d$-balls are open sets. Let $E \subset X$ be a weakly porous set with constants $\gamma, \sigma \in (0,1)$ such that $\mu(\overline{E})=0$ and  $\rho_{d, E}$ is doubling with constant $C_{\rho, E}$. Then, $\dist{\cdot, E}$ is a doubling weight in $(X, d, \mu)$ that belongs to $RH_\infty(X, d, \mu)$ (that is, $\dist{\cdot, E}$ satisfies \eqref{def:RHinf}), with constants depending only on $K_d$, $C_\mu$, $C_{\rho, E}$, $\gamma$, and $\sigma$.  
\end{thm}

\begin{proof} Fix any $M > 2K_d$ (think $M:=3K_d$) and given a ball $B=B(x,r)$ let us consider the following two cases.

\subsection*{Case I: $\dist{x, E} < Mr$} Here we write
\begin{align*}
\fint_{2B} \dist{\cdot, E} \, d\mu \leq \esssup\limits_{2B} \dist{\cdot, E} \leq \esssup\limits_{MB} \dist{\cdot, E},
\end{align*}
with $MB:=B(x,Mr).$ Since $\dist{x, E} < Mr$ we have $MB \cap E \ne \emptyset$ and (since $\mu(\overline{E})=0$) by \eqref{dist:<:rho:Kd},
\begin{equation}\label{sup:d:MB}
\esssup\limits_{MB} \dist{\cdot, E} \leq \rho_{d, E}(MB),
\end{equation}
and, by the doubling property of $\rho_{d, E}$ with constant $C_{\rho, E},$ 
\begin{equation}\label{def:CrhoM}
\rho_{d, E}(MB) \leq C_{\rho, E, M} \, \rho_{d, E}(B)
\end{equation}
for some constant $C_{\rho, E, M}$ depending only on $M$ and $C_{\rho, E}$. Thus,
$$
\fint_{2B} \dist{\cdot, E} \, d\mu \leq C_{\rho, E, M} \, \rho_{d, E}(B),
$$
which followed by \eqref{rho:average:d} gives
$$
\fint_{2B} \dist{\cdot, E} \, d\mu \leq     \frac{C_1 C_{\rho, E, M} }{\gamma \sigma}   \fint_{B} \dist{\cdot, E} \, d\mu,
$$
and, from the fact that $\mu$ is doubling with constant $C_\mu$,
$$
\int_{2B} \dist{\cdot, E} \, d\mu \leq \frac{C_1  C_\mu C_{\rho, E, M} }{\gamma \sigma}   \int_{B} \dist{\cdot, E} \, d\mu,
$$
and $\dist{\cdot, E}$ is therefore doubling on $B$. On the other hand, from \eqref{sup:d:MB}, \eqref{def:CrhoM}, and \eqref{rho:average:d},
\begin{align*}
\esssup\limits_{B} \dist{\cdot, E}  \leq \esssup\limits_{MB} \dist{\cdot, E} \leq \rho_{d, E}(MB)&  \leq  C_{\rho, E, M} \, \rho_{d, E}(B)\\
& \leq \frac{C_1 C_{\rho, E, M}}{\gamma \sigma} \fint_B \dist{\cdot, E} \, d\mu,
\end{align*}
and the $RH_\infty$ reverse-H\"older inequality for $\dist{\cdot, E}$ holds on $B$.

\subsection*{Case II: $\dist{x, E} \geq Mr$} In this case we will prove that 
\begin{equation}\label{dyE:constant}
\left(\frac{1}{K_d} - \frac{2}{M} \right) \dist{x, E} \leq \dist{y, E} \leq K_d (1 + 2/M) \dist{x, E}, \quad \forall y \in B(x, 2r).
\end{equation}
That is, in this case $\dist{\cdot, E}$ behaves like a constant (say,  $\dist{x, E}$) in  $B(x, 2r)$. In particular, the inequalities \eqref{dyE:constant} imply
\begin{align*}
\int_{2B} \dist{y, E} \, d\mu(y) &\leq K_d (1 + 2/M) \dist{x, E}\mu(2B) \leq K_d C_\mu (1 + 2/M) \dist{x, E}\mu(B)\\
& \leq \frac{K_d^2 M  C_\mu (1 + 2/M)}{M- 2K_d} \int_{B} \dist{y, E} \, d\mu(y),
\end{align*}
and $\dist{\cdot, E}$ is then doubling on $B$. Also by \eqref{dyE:constant}
\begin{equation}\label{ineqs:case:2}
\esssup\limits_{y \in B} \dist{y, E}  \leq K_d (1 + 2/M) \dist{x, E} \leq \frac{K_d^2  (M + 2)}{(M-2K_d)} \fint_B \dist{y, E} \, d\mu(y)
\end{equation}
and the $RH_\infty$ reverse-H\"older inequality for $\dist{\cdot, E}$ holds on $B$ in this case as well. 

In order to prove the first inequality in \eqref{dyE:constant}, given $\eps > 0$, let $e_\eps \in E$ such that $d(y, e_\eps) < d(y, E) + \eps$, so that for $y \in B(x, 2r)$ and by keeping in mind that in this case $\dist{x, E} \geq Mr$, we have
\begin{align*}
K_d (d(y, E) + \eps)& > K_d d(y, e_\eps) \geq d(x,e_\eps) - K_d d(y,x) > d(x,e_\eps) - 2 r K_d\\
& \geq d(x,e_\eps) - \frac{2K_d}{M} d(x, E) \geq d(x, E) (1- 2K_d/M),
\end{align*}
and letting $\eps \to 0$ yields the first inequality in \eqref{dyE:constant}. To prove the second inequality in \eqref{dyE:constant}, given $y \in B(x, 2r)$ and $e \in E$, we do
\begin{align*}
d(y, E) \leq d(y, e) \leq K_d d(y,x) + K_d d(x,e) < 2K_d r +  K_d d(x,e) \leq \frac{2K_d}{M} d(x, E) +  K_d d(x,e),
\end{align*}
and the second inequality in \eqref{dyE:constant} follows by taking infimum in $e \in E$.  \end{proof}

We close this section with two observations derived from the above computations. The first one shows that, without the doubling assumption on $\rho_{d,E}$,  the weak porosity of $E$ implies that $\dist{\cdot, E}$ belongs to the weak $RH_\infty$ reverse-H\"older class; while the second one gives a characterization of weak porosity under the doubling assumption on $\rho_{d,E}$. 

\begin{coro}\label{coro:E:WP:wRH} Let $(X, d, \mu)$ be a space of homogenous type with constants $(K_d, C_\mu)$ where the $d$-balls are open sets. Let $E \subset X$ be a weakly porous set with constants $\gamma, \sigma \in (0,1)$ such that $\mu(\overline{E})=0$. Then, $\dist{\cdot, E}$ belongs to the weak $RH_\infty(X, d, \mu)$ class, meaning that there exists a constant $C > 0$, depending only on $K_d$, $C_\mu$, $\gamma$, and $\sigma$, such that 
\begin{equation}\label{dist:E:wRH}
\esssup\limits_{B} \dist{\cdot, E}  \leq C \fint_{MB} \dist{\cdot, E} \, d\mu, \quad \forall d\text{-ball } B,
\end{equation}
where $M:= K_d(1+2K_d)$. 
\end{coro}

\begin{proof} We have $M:= K_d(1+2K_d)$, in particular $M \geq 3K_d$, and given a $d$-ball $B=B(x,r)$ consider Cases I and II as in the proof of Theorem  \ref{thm:RH:doub}. In Case I use \eqref{dist:<:rho:Kd} followed by \eqref{rho:average:d} to obtain
$$
\esssup\limits_{B} \dist{\cdot, E}  \leq \rho_{d, E}(MB) \leq \frac{C_1}{\gamma \sigma} \fint_{MB} \dist{\cdot, E} \, d\mu,
$$
which together with \eqref{ineqs:case:2} from Case II, implies \eqref{dist:E:wRH}. 
\end{proof}

\begin{coro} Let $(X, d, \mu)$ be a space of homogeneous type where the $d$-balls are open sets and the Lebesgue differentiation theorem holds true. Let $E \subset X$ be a nonempty set such that $\rho_{d, E}$ is doubling. Then the following are equivalent:
\begin{enumerate}[(i)]
\item\label{E:wp} $E$ is weakly porous,
\item\label{rho:C:aver:B:dE} $\mu(\overline{E})=0$ and there exists $C >0$ such that
\begin{equation}\label{rho<C:aver:B:dE}
\rho_{d, E}(B) \leq C \fint_B \dist{\cdot, E} \, d\mu, 
\end{equation}
for every $d$-ball with $B \cap E \neq \emptyset$. 
\end{enumerate}
\end{coro}

\begin{proof} $\eqref{E:wp}  \Rightarrow \eqref{rho:C:aver:B:dE}$. The inequality \eqref{rho<C:aver:B:dE} follows from Lemma \ref{lem:Ewp:rho:aver:dE} for any $d$-ball $B$, whereas $\mu(\overline{E})=0$ is due to Remark \ref{rmk:mu:Ebar:0}. For the implication $\eqref{rho:C:aver:B:dE}   \Rightarrow  \eqref{E:wp}$, notice that \eqref{rho:C:aver:B:dE} along with \eqref{dist:<:rho:Kd} and the doubling property of $\rho_{d, E}$ implies 
\begin{equation}\label{coro:1}
\rho_{d, E}(B) \leq C \fint_B \dist{\cdot, E} \, d\mu \leq C \esssup_B \dist{\cdot, E} \leq  CC_2 \, \rho_{d, E}(B)
\end{equation}
for some $C_2 > 1$ (depending only on $K_d$ and the doubling constant of $\rho_{d, E}$) and for every $d$-ball with $B \cap E \neq \emptyset$. Now, fix $M > 2K_d$ and consider the Cases I and II as in the proof of Theorem \ref{thm:RH:doub} to obtain that $\dist{\cdot, E}$ is a doubling weight in $(X, d, \mu)$ that satisfies \eqref{def:RHinf} (by using \eqref{coro:1} instead of \eqref{rho:average:d} in Case I). Then, by Lemma \ref{lemma:RH>Ap} and Remark \ref{rmk:idea:proof} there exists $\alpha > 0$ such that $\dist{\cdot, E}^{-\alpha} \in A_1(X, d, \mu)$, which implies that $E$ is weakly porous by \cite[Theorem 1.1]{AGG1}.\end{proof}

\section{A short proof of Theorem \ref{thm:A1:rho:doub} ($A_1$ implies $\rho_{d, E}$ doubling)}\label{sec:A1:rho:doub}

\begin{lem} Let $(X, d, \mu)$ be a space of homogeneous type where the $d$-balls are open sets. Let $E \subset X$ be a nonempty set such that $\dist{\cdot, E}^{-\alpha} \in A_1(X, d, \mu)$ for some $\alpha >0$. Then, $\dist{\cdot, E}^{\alpha}$ is a doubling weight satisfying  the $RH_\infty(X, d, \mu)$ reverse-H\"older inequality and
\begin{equation}\label{rho:A1:aver:B:dE}
\rho_{d,E}(B) \leq 2K_d[\dist{\cdot, E}^{-\alpha}]_{A_1}^{1/\alpha} \left( \fint_B \dist{\cdot, E}^{\alpha} \, d\mu \right)^{1/\alpha}, \quad \forall d\text{-ball } B. 
\end{equation}
\end{lem}

\begin{proof} By Lemma \ref{lem:basic:Ainf} with  $w:=\dist{\cdot, E}^{-\alpha} \in A_1(X, d, \mu)$, we get that $\dist{\cdot, E}^{\alpha}$ is a doubling weight with constant $C_\mu^2 [\dist{\cdot, E}^{-\alpha}]_{A_1}$ that satisfies the $RH_\infty(X, d, \mu)$ reverse-H\"older inequality with constant $[\dist{\cdot, E}^{-\alpha}]_{A_1}$. Now, fix a $d$-ball $B=B(x,r)$. Since $\dist{\cdot, E}^{-\alpha} \in L^1_{\rm{loc}}(X, d, \mu)$ we must have $\mu(\overline{E})=0$ and therefore $\rho_{d,E}(B) >0$. Then, given any $0 < s < 2K_d r$ and $B(y,s) \subset B\setminus E$, by the inequality \eqref{dzE:tau} applied with $C=2K_d$ we obtain
$$
\dist{z, E} \geq \frac{s}{2K_d}, \quad \forall z \in B(y, s/K_d),
$$
which combined with \eqref{fint:F:fint:B:u} used with $u := \dist{\cdot, E}^{\alpha}$ and $F := B(y, s/K_d)$ gives
$$
\left(\frac{s}{2K_d} \right)^\alpha \leq \fint_F \dist{\cdot, E}^{\alpha} \, d\mu \leq [\dist{\cdot, E}^{-\alpha}]_{A_1} \fint_B \dist{\cdot, E}^{\alpha} \, d\mu
$$
and  \eqref{rho:A1:aver:B:dE} follows by taking supremum in $s$. \end{proof}

\subsection*{Proof of Theorem \ref{thm:A1:rho:doub}} Set $K':= 2K_d(1+2K_d)$ (notice $2/K' = 1/[K_d(1+2K_d)] \leq 1/3$), and given a ball $B=B(x, r)$ set $B':=B(x, r/K')$. Assume that $\rho_{d,E}(B) >0$ (otherwise \eqref{doub:rho} follows trivially) and consider the following three cases: 
\subsection*{Case I: $B' \cap E \ne \emptyset$} Then, by \eqref{dist:<:rho:Kd} applied to $B'$, 
\begin{equation}\label{dzE:B'}
\esssup\limits_{B'} \dist{\cdot, E}  \leq \rho_{d, E}(B') \leq \rho_{d, E}(B(x, r/2)).
\end{equation}
(Notice that $\dist{\cdot, E}^{-\alpha} \in A_1(X, d, \mu)$ implies $\dist{\cdot, E}^{-\alpha} \in L^1_{\rm{loc}}(X, d, \mu)$; in particular, $\mu(\overline{E}) =0$.) By the doubling property of $\dist{\cdot, E}^{\alpha}$ there is a constant $C_3 > 0$ (depending only on $C_\mu$,  $[\dist{\cdot, E}^{-\alpha}]_{A_1}$, and $K_d$) such that
$$
\int_{B} \dist{\cdot, E}^{\alpha} \, d\mu \leq C_3 \int_{B'} \dist{\cdot, E}^{\alpha},
$$
which together with \eqref{rho:A1:aver:B:dE} and \eqref{dzE:B'} gives 
\begin{align*}
\rho_{d,E}(B) & \leq 2K_d[\dist{\cdot, E}^{-\alpha}]_{A_1}^{1/\alpha} \left( \fint_B \dist{\cdot, E}^{\alpha} \, d\mu \right)^{1/\alpha}\\
& \leq 2K_d (C_3 [\dist{\cdot, E}^{-\alpha}]_{A_1})^{1/\alpha} \left( \fint_{B'} \dist{\cdot, E}^{\alpha} \, d\mu \right)^{1/\alpha} \\
& \leq 2K_d (C_3 [\dist{\cdot, E}^{-\alpha}]_{A_1})^{1/\alpha} \rho_{d, E}(B(x, r/2)),
\end{align*}
and \eqref{doub:rho} holds true in this case. 
\subsection*{Case II: $B' \cap E = \emptyset$ and $B \cap E \ne \emptyset$} Fix $z \in B \cap E$ to obtain
$$
\dist{y, E} \leq d(y, z) \leq K_d d(x, z) + K_d d(x, y) < 2K_d r, \quad \forall y \in B.
$$
Also, since $B' =B(x, r/K') \subset B(x, 2r/K') \setminus E$ and $0 < r/K' < 2K_d (2r/K')$, it follows that $r/K' \leq \rho_{d,E}(B(x, 2r/K'))$. Hence,
$$
\sup\limits_{B} \dist{\cdot, E} \leq 2K_d K' \, \rho_{d,E}(B(x, 2r/K')),
$$
which combined with \eqref{rho:A1:aver:B:dE}, gives 
$$ 
\rho_{d,E}(B) \leq C_4  \, \rho_{d,E}(B(x, 2r/K'))
$$ 
for a constant $C_4 \geq 1$ depending only on $K_d$, $\alpha$, and $[\dist{\cdot, E}^{-\alpha}]_{A_1}$.

\subsection*{Case III: $B \cap E = \emptyset$} Here we notice that $\tfrac{1}{2}B = \tfrac{1}{2}B \setminus E$ with $0 < r/2 < 2 K_d (r/2)$, and then $r/2 \leq \rho_{d, E}(\tfrac{1}{2}B)$. On the other hand, since $\rho_{d, E}(B) >0$ we get $\rho_{d, E}(B) \leq 2K_d r.$ Thus, $\rho_{d, E}(B) \leq 2K_d r \leq 4  K_d \rho_{d, E}(\tfrac{1}{2}B).$\qed

\begin{remark} Notice that the proof of Theorem \ref{thm:A1:rho:doub} only required \eqref{rho:A1:aver:B:dE}, \eqref{dist:<:rho:Kd} and the doubling property of $\dist{\cdot, E}^\alpha$. Thus, the doubling property of $\rho_{d, E}$ can be proved similarly by replacing \eqref{rho:A1:aver:B:dE} with \eqref{rho:average:d} (think $\alpha =1$). This observation along with Theorem \ref{thm:RH:doub} yields then the following corollary.
\end{remark}

\begin{coro}\label{coro:dE:r:doub} Let $(X, d, \mu)$ be a space of homogeneous type where the $d$-balls are open sets and let $E \subset X$ be a weakly porous set with $\mu(\overline{E}) = 0$. Then, $\dist{\cdot, E}$ is doubling if and only if $\rho_{d, E}$ is so.
\end{coro}  

\section{A conjecture}

Let us close this note by venturing the following conjecture: \emph{Let $(X, d, \mu)$ be a space of homogenous type where the $d$-balls are open sets, then  a nonempty set  $E \subset X$ with $\mu(\overline{E})=0$ is weakly porous if and only if $\dist{\cdot, E} \in RH_\infty(X, d, \mu)$, that is, if and only if \eqref{def:RHinf} is true.} 

\begin{remark}\label{rmk:conj:1} Notice that the implication ``$\dist{\cdot, E} \in RH_\infty(X, d, \mu) \Rightarrow$ $E$ weakly porous'' is true in any space of homogenous type $(X, d, \mu)$ where the $d$-balls are open sets, where the Lebesgue differentiation theorem holds true, and where every weight in $RH_\infty(X, d, \mu)$ is a doubling weight. This follows from Lemma \ref{lemma:RH>Ap} and Remark \ref{rmk:idea:proof}. 
\end{remark}

\begin{remark} Let $(X, d, \mu)$ be a space of homogenous type where the $d$-balls are open sets. The measure $\mu$ is said to possess the \emph{annular decay property} if there are constants $C >0$ and $\theta \in (0,1]$ such that
$$
\mu(B \setminus (1-\delta) B) \leq C \delta^\theta \mu(B), \quad \forall d\text{-ball } B, \delta \in (0,1).
$$
The above conjecture is true in the aforementioned context of Mudarra's work if $\mu$ has the annular decay property, since by \cite[Lemma 7.2]{Mud25} the weak porosity of $E$ then implies the doubling property for $\rho_{d, E}$, and then that of $\dist{\cdot, E}$ by Corollary \ref{coro:dE:r:doub}. Thus, $\dist{\cdot, E} \in RH_\infty(X, d, \mu)$ due to Corollary \ref{coro:E:WP:wRH}. Conversely, the implication ``$\dist{\cdot, E} \in RH_\infty(X, d, \mu) \Rightarrow$ $E$ weakly porous'' follows from Remark \ref{rmk:conj:1} and the fact every weight $u \in RH_\infty(X, d, \mu)$ is a doubling weight whenever $\mu$ has the annular decay property. This last claim is a consequence of \eqref{fint:F:fint:B:u}, by taking $F:=B \setminus (1-\delta) B$ and choosing $\delta=\delta(C, \theta, [u]_{RH_\infty})$ close to 1. 
\end{remark}

\section*{Acknowledgements} The author would like to express his gratitude to Hugo Aimar and  Ignacio G\'omez-Vargas from the Instituto de Matem\'atica Aplicada del Litoral ``Dra. Eleonor Harboure'', CONICET, UNL, IMAL, Argentina, for their detailed reading of the manuscript, their generous and insightful comments, and for providing valuable bibliographic references.

\end{document}